\documentclass[centertags,12pt]{amsart}

\usepackage{amssymb}
\usepackage{latexsym}
\usepackage{amsfonts}
\usepackage{amssymb}
\usepackage{graphicx} 
\usepackage{pdflscape} 
\usepackage{diagbox} 
\usepackage{hyperref}

\makeatletter
\renewcommand{\subsection}{\@startsection
{subsection}{2}{0mm}{\baselineskip}{-0.25cm}
{\normalfont\normalsize\em}}
\makeatother

\newtheorem{theorem}{Theorem}[section]
\newtheorem{proposition}[theorem]{Proposition}
\newtheorem{corollary}[theorem]{Corollary}
\newtheorem{lemma}[theorem]{Lemma}

{\theoremstyle{definition}

\newtheorem{example}[theorem]{Example}
\newtheorem{Ex2.1}[theorem]{Example 2.5 revisited}
\newtheorem{Ex2.2}[theorem]{Example 2.6 revisited}
\newtheorem{question}[theorem]{Question}
}

{\theoremstyle{remark}
\newtheorem{remark}[theorem]{Remark}

}

\def\1{\mathbf 1}

\def\Z{\mathbf Z}

\def\N {\mathbb{N}}

\def\a {\alpha}

\def\k {\kappa}
\def\l {\ell}

\DeclareMathOperator{\Imm}{Im}

\title[On Pure $\k$-sparse gapsets]{On pure $\k$-sparse gapsets}

\author[G. B. Almeida Filho]{Gilberto B. Almeida Filho}
\address{}
\email{gbrito.af19@gmail.com}

\author[M. Bernardini]{Matheus Bernardini}
\address{}
\email{matheusbernardini@unb.br}

\thanks{{\em 2020 Math. Subj. Class.}: Primary 20M14; Secondary 05A15, 05A19}


\begin{document}


\maketitle

\begin{center}
{\sc Dedicated to the memory of Fernando Torres}
\end{center}

\begin{abstract}
In this paper, we study gapsets and we focus on obtaining information on how the maximum distance between two consecutive elements influences the behaviour of the set. In particular, we prove that the cardinality of the set of gapsets with genus $g$ such that the maximum distance between two consecutive elements is $\k$ is equal to the cardinality of the set of gapsets with genus $g+1$ such that the maximum distance between two consecutive elements is $\k+1$, when $2g \leq 3\k$.
\end{abstract}

\section{Introduction}
\label{s1}
A \textit{gapset} is a finite set $G \subset \N$ satisfying the following property. Let $z \in G$ and write $z = x + y$, with $x$ and $y \in \N$. Then $x \in G$ or $y \in G$. This concept was formally introduced by Eliahou and Fromentin \cite{EF}. Recall that a numerical semigroup $S$ is a submonoid of the set of non-negative integers $\N_0$, equipped with the usual addition, such that $G(S):=\N_0\setminus S$, the set of {\em gaps} of $S$, is finite (the book \cite{GS-R} is an excellent background on this subject). Hence, a gapset is nothing more than the set of gaps of some numerical semigroup. Throughout this paper, we denote $[a,b] := \{x \in \Z: a \leq x \leq b\}$ and $[a, \infty) := \{x \in \Z: x \geq a\}$, for integers $a$ and $b$.

The cardinality of a gapset $G$ is called the {\em genus} of $G$. Numerical semigroup properties imply that if $S$ is a numerical semigroup with genus $g$, then $[2g, \infty) \subset S$. Hence, a non-empty gapset $G$ with genus $g$ satisfies $G \subseteq [1,2g-1]$. Some invariants play an important role in this theory. For instance, the multiplicity, the conductor, the depth and the Frobenius number of a gapset are $m(G) := \min\{s \in \N: s \notin G\}$, $c(G) := \min\{s \in \N: s + n \notin G, \forall n \in \N_0\}, q(G) := \lceil \frac{c(G)}{m(G)} \rceil$ and $F(G) := c(G) - 1$, respectively.

The set of gapsets is denoted by $\Gamma$ and the set of gapsets with a fixed genus $g$ by $\Gamma(g)$, which has $n_g$ elements. The first few elements of the sequence $(n_g)$ are $1, 1, 2, 4, 7, 12, 23, 39, 67, 118$ and it is registered as the sequence A007323 at OEIS (the on-line encyclopedia of integer sequences). This sequence has been deeply studied after Bras-Amor\' os \cite{Amoros1} conjectured three items about its behaviour. The only one which is still an open question is: ``is it true that $n_g + n_{g+1} \leq n_{g+2}$, for all $g$?". Even the weaker version of this conjecture, namely ``is $(n_g)$ a non-decreasing sequence?" is an open problem. The other two questions are about the asymptotic behaviour of $(n_g)$, namely $\lim_{g \to \infty} \frac{n_{g+1}}{n_g} = \frac{1+\sqrt{5}}{2}$ (golden ratio) and  $\lim_{g \to \infty} \frac{n_{g+1} + n_g}{n_{g+2}} = 1$; both were proved by Zhai \cite{Zhai}. The main tool to prove those results is the fact that most of gapsets with fixed genus have depth at most 3. More precisely:

\paragraph{Zhai's theorem.}
Let $n_g' := \#\{G \in \Gamma(g): q(G) \leq 3\} $. Then
$$\lim_{g \to \infty} \frac{n_g'}{n_g} = 1.$$
For more details on this subject, we recommend the survey written by Kaplan \cite{Kaplan2}.

Some authors studied the set of numerical semigroups with fixed genus and some other invariants. For instance, Bernardini and Torres \cite{MF} studied numerical semigroups with fixed genus and fixed number of even gaps; Blanco and Rosales \cite{BR} studied numerical semigroups with fixed genus and fixed Frobenius number; Bras-Amor\' os \cite{Amoros3} studied numerical semigroups with fixed genus and fixed ordinarization number; Kaplan \cite{Kaplan} studied numerical semigroups with fixed genus and fixed multiplicity; all of them obtained some interesting partial results.

Throughout this paper, we write a gapset $G$ with genus $g$ as $G = \{\l_1 < \l_2 < \ldots < \l_g\}$ and its elements will always be enumerated in the increasing natural order. In this paper, we study gapsets with fixed genus and we focus on pure $\k$-sparse gapsets. A $\k$-sparse gapset $G$ is a gapset such that the distance between two consecutive elements (in the natural order) is at most $\k$, i.e., $\l_{j+1} - \l_j \leq \k$, for all $j$. A pure $\k$-sparse gapset $G$ is a $\k$-sparse gapset such that there are two consecutive elements $\l_i$ and $\l_{i+1}$ such that $\l_{i+1} - \l_i = \k$. By convention, we say that $\emptyset$ is a pure $0$-sparse gapset and $\{1\}$ is a pure $1$-sparse gapset. The set of pure $\k$-sparse gapsets with genus $g$ is denoted by $\mathcal{G}_\k(g)$.

Munuera, Torres and Villanueva \cite{MTV} studied the so-called sparse semigroups, which are numerical semigroups such that its pairs of consecutive gaps have distance at most 2. In particular, they prove that the set of Arf semigroups is a proper subset of the set of sparse semigroups. Tizziotti and Villanueva \cite{JG} studied $\k$-sparse and pure $\k$-sparse numerical semigroups, which naturally generalize the concept of sparse semigroups. Here, we choose using gapset theory because we use information about the set of gaps of a numerical semigroup, such as its cardinality, the maximum distance between two consecutive gaps and its depth.

Throughout, we use the following notation: if $G$ is a pure $\k$-sparse gapset with genus $g \geq 2$, we write $G = \{1 = \l_1 < \l_2 < \cdots < \l_\a < \l_{\a+1} = \l_\a + \k < \cdots < \l_g\}$, with the index $\a$ being the maximum element of the set $\{i: \l_{i+1} - \l_i = \k\}$.

Here is an outline of this paper. In section 2, we present some general properties of gapsets. We show how  the multiplicity, the conductor and the depth of a gapset are related to its genus. In section 3, we study pure $\k$-sparse gapsets and we check some properties in relation to the elements of a pure $\k$-sparse gapset. In this case, we prove that the maximum distance between the consecutive elements can be limited by its multiplicity, its genus and its conductor. In addition, when we have the hypothesis $2g \leq 3\k$, where $g$ is the genus of the gapset and $\k$ is the biggest distance between two consecutive elements, the gapset has exactly one pair of consecutive elements with distance $\k$ (except for one case) and it has depth at most 3. In section 4, we build an injective map $\phi$ (see equation (\ref{mapaphi}), in page 8) from the family $\mathcal{G}_\k(g)$ to the set of subsets of $[1, 2g + 1]$ and we look for conditions on $g$ and $\k$ so that the image of a gapset is again a gapset. In particular, we prove that every gapset with depth at most two and every pure $\k$-sparse gapset with genus $g$ such that $2g \leq 3\k$ satisfy this property. We emphasize that this map brings a new approach to this topic. In section 5, we prove that the application $\varphi: \mathcal{G}_\k(g) \to \mathcal{G}_{\k+1}(g+1), G \mapsto \varphi(G) = \phi(G)$ is surjective and therefore establishes a bijection between $\mathcal{G}_{\k}(g)$ and $\mathcal{G}_{\k+1}(g+1)$ whenever the condition $2g \leq 3\k$ is satisfied. Thus, we conclude that $\# \mathcal{G}_{k}(g) = \# \mathcal{G}_{k+n}(g+n)$, for all $n \in \N$. In section 6, we discuss some further questions related to this work.

\section{Some basic results on gapsets}
\label{s2}
We bring some familiar results on numerical semigroup theory to gapset theory and we keep the proofs for the sake of completeness. The next result tells us about the relation between multiplicity and genus of a gapset.

\begin{proposition}
If a non-empty gapset has genus $g$ and multiplicity $m$, then $2 \leq m \leq g+1$.
\label{mg}
\end{proposition}

\begin{proof}
If $m \geq g+2$, then $[1,g+1] \subseteq G$ and $\# G \geq g+1$, which does not occur. Hence $m \leq g+1$. On the other hand, since $G \neq \emptyset$, then $1 \in G$ and $m \geq 2$.
\end{proof}

The next result gives us information about each element of a gapset and it can be found in \cite{Gilvan}.

\begin{proposition}
Let $G = \{\l_1 < \l_2 < \cdots < \l_g\}$ be a gapset with genus $g$. Then $j \leq \l_j \leq 2j-1$.
\label{gilvan}
\end{proposition}

\begin{proof}
If $\l_j < j$ for some $j$, then $\l_1 < 1$, which is a contradiction. On the other hand, let $j \in [1,g]$. If $\l_j$ is even, then $\frac{\l_j}{2} \in G$ and for each $\l \in [1, \frac{\l_j}{2} - 1]$, there is at least one element of $G$ in $\{\l, \l_j - \l\}$. There are $j$ elements of $G$ in $[1, \l_j]$, including $\frac{\l_j}{2}$ and $\l_j$, hence $j \geq (\l_j/2 - 1) + 2$ and we obtain $\l_j \leq 2j - 2 < 2j - 1$.  If $\l_j$ is odd, then for each $\l \in [1, \frac{\l_j-1}{2}]$, there is at least one element of $G$ in $\{\l, \l_j - \l\}$. There are $j$ elements of $G$ in $[1, \l_j]$, including $\l_j$, hence $j \geq \frac{\l_j-1}{2} + 1$ and we obtain $\l_j \leq 2j - 1$.
\end{proof}

The next result tells us about the relation between the conductor and the genus of a gapset.

\begin{proposition}
If a non-empty gapset has genus $g$ and conductor $c$, then $g+1 \leq c \leq 2g$.
\label{cond}
\end{proposition}

\begin{proof}
If $c \leq g$, then $G \subseteq [1,g-1]$ and $\# G \leq g-1$, which does not occur. Hence, $c \geq g+1$. On the other hand, Proposition \ref{gilvan} ensures that $\l_g \leq 2g - 1$ and thus $c \leq 2g$.
\end{proof}

As a consequence, we obtain that if a non-empty gapset has genus $g$ and Frobenius number $F$, then $g \leq F \leq 2g-1$. Putting together Propositions \ref{mg} and \ref{cond}, we obtain a relation between the depth and the genus of a gapset.

\begin{corollary}
If a non-empty gapset has genus $g$ and depth $q$, then $1 \leq q \leq g$.
\label{dep}
\end{corollary}

\begin{proof}
By applying Propositions \ref{mg} and \ref{cond}, we obtain
$$1 = \frac{g+1}{g+1} \leq \frac{c}{m} \leq \frac{2g}{2} = g.$$
Hence, $1 \leq q \leq g$.
\end{proof}

Next, we present three classical examples. In particular, Examples \ref{q=1} and \ref{q=g} confirm that the bounds presented in Propositions \ref{mg}, \ref{gilvan}, \ref{cond} and in Corollary \ref{dep} are accurate.

\begin{example}
The trivial gapset $\emptyset$ has genus $0$, multiplicity $1$, conductor $0$ and depth $0$. It is the unique gapset with depth 0.
\label{q=0}
\end{example}

\begin{example}
Let $g \in \N$. The gapset $[1,g]$ has genus $g$, multiplicity $g+1$, conductor $g+1$ and depth $1$. There are no other gapsets with depth 1. Following the numerical semigroup terminology, those gapsets are \textit{ordinary gapsets}.
\label{q=1}
\end{example}

\begin{example}
Let $g \in \N$. The gapset $(2\N + 1) \cap [1,2g-1]$ has genus $g$, multiplicity $2$, conductor $2g$ and depth $g$. There are no other gapsets with depth $g$. Following numerical semigroup terminology, those gapsets are \textit{hyperelliptic gapsets}.
\label{q=g}
\end{example}

We end up this section with the following result.

\begin{proposition}
Let $G = \{\l_1 < \l_2 < \ldots < \l_g\}$ be a gapset with multiplicity $m$ and genus $g$. Then $[am + \l_j + 1, am + \l_{j+1} - 1] \cap G = \emptyset$ for all $a \in \N_0$ and for all $j \in [1,g-1]$.
\label{gilberto_partition}
\end{proposition}

\begin{proof}
Let $a \in \N_0$ and $j \in [1,g-1]$. If $b \in [am + \l_j + 1, am + \l_{j+1} - 1]$, then  $b = am + \l_j + c$, with $c \in [1, \l_{j+1} - \l_j - 1]$. Since $am \notin G$ and $\l_j + c \notin G$, then $b \notin G$.
\end{proof}

\section{Pure $\k$-sparse gapsets}
\label{s3}

Following the notation for numerical semigroups, a gapset $G$ is \textit{$\k$-sparse} if the difference between two consecutive elements of $G$ (in the natural order) is at most $\k$. If there are two consecutive elements $\l_{i}$ and $\l_{i+1} \in G$ such that $\l_{i+1} - \l_i = \k$, then we say that $G$ is a pure $\k$-sparse gapset. Let $\mathcal{G}_{\k}$ be the set of all pure $\k$-sparse gapsets and $\tilde{\mathcal{G}}_{\k}$ be the set of all $\k$-sparse gapsets. In particular,  $\mathcal{G}_{\k} \subset \tilde{\mathcal{G}_{\k}}$ and $\tilde{\mathcal{G}_{\k}} \subset \tilde{\mathcal{G}}_{\k+1}$. In this section, we study the set of pure $\k$-sparse gapsets with genus $g$, namely $\mathcal{G}_{\k}(g) := \Gamma(g) \cap \mathcal{G}_{\k}$ and its subset $\mathcal{G}_{\k}(g, q) := \{G \in \mathcal{G}_{\k}(g): q(G) = q\}$. By convention, we say that $\emptyset$ is a $0$-sparse gapset and $\{1\}$ is a $1$-sparse gapset. We recall that if $G$ is a pure $\k$-sparse gapset with genus $g \geq 2$, then we write $G = \{\l_1 < \l_2 < \cdots < \l_\a < \l_{\a+1} = \l_\a + \k < \cdots < \l_g\}$, where $\a = \max\{i: \l_{i+1} - \l_i = \k\}$.

\begin{example}
Let $G = \{\l_1 < \l_2 < \ldots < \l_g\}$ be a pure $1$-sparse gapset. Then $\l_{i+1} - \l_i = 1$, for all $i$. Hence, $G = \{1, 2, \ldots, g\}$ is an ordinary gapset.
\end{example}

Now we obtain a relation between the multiplicity of a gapset and the maximum distance between two consecutive elements.

\begin{proposition}
Let $G$ be a pure $\k$-sparse gapset with multiplicity $m$. Then $\k \leq m$.
\label{km}
\end{proposition}

\begin{proof}
Let $\l_i$ and $\l_{i+1} \in G$ such that $\l_{i+1} - \l_i = \k$. If $\k > m$, then $\l_i + 1, \l_i + 2, \ldots, \l_i + m \notin G$ ($\l_i + m < \l_{i+1}$). There is $a \in \Z$ such that $\l_i + 1 \leq \l_{i+1} - am \leq \l_i + m$.  Notice that $G \ni \l_{i+1} = (\l_{i+1} - am) + am$. But $\l_{i+1} - am \notin G$ and $am \notin G$, which leads to a contradiction.
\end{proof}

At first, Proposition \ref{km} together with Proposition \ref{mg} imply that a pure $\k$-sparse gapset with genus $g$ satisfies $\k \leq g + 1$. We obtain a better result in the next proposition.

\begin{proposition}
Let $G$ be a pure $\k$-sparse gapset with genus $g$. Then $\k \leq g$.
\end{proposition}

\begin{proof}
Let $\l_i$ and $\l_{i+1} \in G$ such that $\l_{i+1} - \l_i = \k$. If $\k \geq g+1$, then $\l_i + 1, \l_i + 2, \ldots, \l_i + g \in [1,2g-2] \cap (\Z \setminus G)$ ($\l_i + g < \l_{i+1} \leq \l_{g}$). Hence, \linebreak $\# G \leq (2g - 2 - g) + 1 = g - 1$, which is a contradiction.
\end{proof}

The next result improves the upper bound on Proposition \ref{cond}.

\begin{proposition}
Let $G$ be a pure $\k$-sparse gapset with conductor $c$ and genus $g$. Then $g + \k \leq c$.
\label{kF}
\end{proposition}

\begin{proof}
Suppose that $g + \k > c$. By writing $G = \{\l_1 < \l_2 < \cdots < \l_\a < \l_{\a+1} = \l_\a + \k < \cdots < \l_g\}$, we conclude that $\l_g \leq g + \k - 2$. We have that \linebreak $I = [\l_\a + 1, \l_{\a+1} - 1] \cap G = \emptyset$, $\#I = \k-1$ and $G \subset [1, g + \k - 2]$. Hence, $\#G \leq (g+ \k -2) - (\k-1) = g - 1$, which is a contradiction.
\end{proof}

\begin{example}
The ordinary gapset with genus $g$ has multiplicity $g+1$, conductor $g+1$ and $\k = 1$. The hyperelliptic gapset  with genus $g$ has multiplicity $2$, conductor $2g$ and $\k = 2$. Notice that they do not attain the maximum $\k$ with respect to the multiplicity and the genus. The ordinary gapset attains the bound $g + \k = c$. The gapset $G = \{1, 2, \ldots, g-1, 2g-1\}$ has genus $g$, multiplicity $g$ and $\k = g$, i. e., it attains the maximum value of $\k$ with respect to the multiplicity and with respect to the genus. 
\end{example}

\begin{remark}
If $G = \{\l_1 < \l_2 < \ldots < \l_g\}$ is a pure $\k$-sparse gapset with genus $g$ such that $\l_{i+1} - \l_i = \k$, then $\k \leq i+1$. This follows from Proposition \ref{gilvan}, since $\k = \l_{i+1} - \l_i \leq 2i+1 - i = i+1$.
\end{remark}

The next result is important for the proofs of the main results of this paper. It confirms that the Frobenius number of a gapset cannot be too big with respect to its element $\l_\a$.

\begin{proposition}
Let $G = \{\l_1 < \l_2 < \cdots < \l_\a < \l_{\a+1} = \l_\a + \k < \cdots < \l_g\}$ be a pure $\k$-sparse gapset with genus $g$, where $\a = \max\{i: \l_{i+1} - \l_i = \k\}$. Then $\l_g \leq \l_\a + m$.
\label{abaixo}
\end{proposition}

\begin{proof}
Suppose that $\l_g > \l_\a + m$ and let $r$ be the smallest index such that $\l_r > \l_\a + m$. By Proposition \ref{gilberto_partition}, we conclude that $\l_r \geq \l_{\a+1} + m$ and thus $\l_r - \l_{r-1} \geq \k$. Since $G$ is a pure $\k$-sparse gapset, we have $\l_r - \l_{r-1} = \k$ and thus $r \geq \alpha + 1$. On the other hand, $r \leq \alpha + 1$. In fact, we know that $\l_j \leq \l_{\a} < \l_{\a} + m$, for all $j \in [1,\a]$. By the definition of $\a$, we know that $\a$ is the greatest index such that $\l_{i+1} - \l_i = \k$ and thus $r \leq \a +1$. Hence, the only possibility is $r = \a + 1$. However $\l_{\a+1} = \l_\a + \k \leq \l_\a + m$ and we obtain a contradiction. Hence, $\l_g \leq \l_\a + m$.
\end{proof}

Recall that the canonical partition of a gapset $G$ with multiplicity $m$ and depth $q$ is $G_0 \cup G_1 \cup \ldots \cup G_{q-1}$, where $G_0 =[1, m-1]$ and $G_i = G \cap [im + 1, (i+1)m - 1]$ (cf. \cite{EF}). As a consequence of Proposition \ref{abaixo}, we obtain information about where the element $\l_\a$ can be. 


\begin{corollary}
Let $G = \{\l_1 < \l_2 < \cdots < \l_\a < \l_{\a+1} = \l_\a + \k < \cdots < \l_g\}$ be a pure $\k$-sparse gapset with genus $g$, where $\a = \max\{i: \l_{i+1} - \l_i = \k\}$, and consider its canonical partition $G_0 \cup G_1 \cup \cdots \cup G_{q-1}$. Then one of the following occurs:
  \begin{itemize}
    \item[$\bullet$] $\l_\a, \l_{\a+1} \in G_{q-2}$
    \item[$\bullet$] $\l_\a, \l_{\a+1} \in G_{q-1}$
    \item[$\bullet$] $\l_\a \in G_{q-2}$ and $\l_{\a+1} \in G_{q-1}$
  \end{itemize}
\label{trescasos}
\end{corollary}

\begin{proof}
From Proposition \ref{abaixo}, we conclude that $\l_\a \in G_{q-2}$ or  $\l_\a \in G_{q-1}$, because $\l_g \in G_{q-1}$. For the first case, both possibilities can occur: $\l_{\a+1} \in G_{q-2}$ or $\l_{\a+1} \in G_{q-1}$. For the second case, we must have $\l_{\a+1} \in G_{q-1}$ and we are done.
\end{proof}

The next example illustrates that the three cases obtained in Corollary \ref{trescasos} can occur. For emphasizing, the last pair of consecutive elements of a gapset that attains the maximum distance will be written in bold.

\begin{example}
For an integer $m \geq 3$, we have the following examples:
  \begin{itemize}
    \item[$\bullet$]  $\{1, \ldots, m-1, \boldsymbol{m + 1}, \boldsymbol{2m - 1}, 2m + 1\} = [1,m-1] \cup \{m+1, 2m-1\} \cup \{2m+1\}$ has $q = 3$, $\l_\a = m+1 \in G_1$ and $\l_{\a+1} = 2m-1 \in G_1$;
    \item[$\bullet$]  $\{1, \ldots, m-1, \boldsymbol{m + 1}, \boldsymbol{2m - 1}\} = [1,m-1] \cup \{m+1, 2m-1\}$ has $q = 2$, $\l_\a = m + 1 \in G_1$ and $\l_{\a+1} = 2m - 1 \in G_1$;
    \item[$\bullet$]  $\{1, \ldots, m-1, \boldsymbol{m + 1}, \boldsymbol{2m + 1}\} = [1, m-1] \cup \{m+1\} \cup \{2m + 1\}$ has $q = 3$, $\l_\a = m+1 \in G_1$ and $\l_{\a+1} = 2m + 1 \in G_2$.
  \end{itemize}
\end{example}

\noindent \textbf{The case $2g \leq 3\kappa$.} At this point, we focus on pure $\k$-sparse gapsets with genus $g$ such that $2g \leq 3\k$ which play an important role in the end of Section 5 and in Section 6.

\begin{proposition}
Let $G$ be a pure $\k$-sparse gapset with genus $g$ and depth $q$. If $2g \leq 3\k$, then $q \leq 3$.
\label{2g3kq3}
\end{proposition}

\begin{proof}
Since $m \geq \k \geq \frac{2g}{3}$ and $c \leq 2g$, we have $q = \left \lceil \frac{c}{m} \right \rceil \leq \frac{2g}{\frac{2g}{3}} = 3$.
\end{proof}

\begin{remark}
It is important to notice that the gapset $\{1, 2, 3, 4, \textbf{6}, \textbf{9}, 11\}$ has depth $3$, genus $7$ and $\k = 3$. In this case $2g > 3\k$. Hence, the converse of Proposition \ref{2g3kq3} is not true.
\end{remark}

The next proposition tells us that, if $2g \leq 3\k$, then the set $\{i: \l_i - \l_{i-1} = \k\}$ has only the element $\a$, with only one exception.

\begin{proposition}
Let $G$ be a pure $\k$-sparse gapset with genus $g$ such that $2g \leq 3\k$. Then there is an unique $i \in [1, g-1]$ such that $\l_{i+1} - \l_i = \k$, except for $G = \{1,3,5\}$.
\label{unique}
\end{proposition}

\begin{proof}
Suppose that there are $i \neq j$ such that $\l_{i+1} - \l_i = \l_{j+1} - \l_j = \k$. Then $\l_g \geq 1 + 2 \cdot \k + (g - 3) \cdot 1 = 2\k + g - 2 \geq 4g/3 + g - 2 = 7g/3 - 2 \geq 2g$, if $g \geq 6$. By Proposition \ref{gilvan}, it is a contradiction. The cases $g \leq 5$ can be done by using GAP.
\end{proof}

\begin{remark}
If $2g > 3\k$, then it is possible to have a unique $i$ such that $\l_{i+1} - \l_i = \k$ and it is possible to have more than one index that satisfies that property. For instance, if $G_1 = \{1,2,3,5,6\}$, then $2g(G_1) = 10 > 3\k(G_1) = 6$ and $\k(G_1) = 5 - 3$; if $G_2$ is the hyperelliptic gapset with genus $g \geq 4$, then $2g(G_2) \geq 8 > 3\k(G_1) = 6$ and every pair of consecutive elements of $G_2$ have distance 2.
\end{remark}

For the next result we consider a pure $\k$-sparse gapset $G$ with genus $g$ such that $2g \leq 3\k$ and we use the last result about the uniqueness of a pair of consecutive elements with the greatest distance $\k$. If $q(G) = 3$ and $G_0 \cup G_1 \cup G_2$ is the canonical partition of $G$, then Corollary \ref{trescasos} ensures that $\l_\a \in G_1$ or $\l_a \in G_2$. The next result ensures that $\l_\a \in G_1$ under the hypothesis $2g \leq 3k$.

\begin{proposition}
Let $g > 1$ and $\k$ be positive integers such that $2g \leq 3\k$ and consider $G = \{\l_1 < \l_2 < \ldots < \l_\a < \l_{\a+1} = \l_\a + \k < \ldots < \l_g\}$ be a pure $\k$-sparse gapset with multiplicity $m$,  where $\a = \max\{i: \l_{i+1} - \l_i = \k\}$. Then $\l_\a \leq 2m - 1$.
\label{gilberto}
\end{proposition}

\begin{proof}
The hypothesis $2g \leq 3\k$ ensures that the depth of $G$ is at most 3. Suppose that $\l_\a \geq 2m+1$. By Proposition \ref{km}, $\l_{\a+1} = \l_\a + \k \geq 2m + 1 + \k \geq 3\k +1$. Hence, $\l_{\a+1} \geq 2g + 1$, which is a contradiction.
\end{proof}

\begin{remark}
This property does not hold true in the general case, even for depth at most 3. For instance, the gapset $G = \{1,2,3,4,6,7,9,\textbf{11},\textbf{14}\}$ has $\l_\a = 11 > 9 = 2m-1$.
\end{remark}

\begin{remark}
Proposition \ref{gilberto} gives informations about $\l_\a$, but not about $\l_{\a+1}$. For instance, if $G = \{1, 2, \ldots, 10, \textbf{12}, \textbf{21}, 23\}$, then $\l_{\a+1} \leq 2m-1$; if $G = \{1,2,\textbf{4},\textbf{7}\}$, then $\l_{\a+1} > 2m-1$.
\label{2casos}
\end{remark}

\section{The map $\phi$}
\label{s4}

In this section, we introduce the map $\phi$ and we discuss which properties the set $\phi(G)$ has, when $G \in \mathcal{G}_\k(g)$. Our aim is that $\phi(G)$ is a pure $(\k+1)$-sparse gapset with genus $g+1$. Hence, we look for hypotheses that guarantee it. 

Eliahou and Fromentin \cite{EF} introduced the notion of $m$-extension, for $m \in \N$. An $m$-extension $A \subset \N$ is a finite set containing $[1,m-1]$ that admits a partition \linebreak $A = A_0 \cup A_1 \cup \ldots \cup A_t$, for some $t \in \N_0$, where $A_0 = [1,m-1]$ and $A_{i+1} \subseteq m + A_i$ for all $i$. In particular, if $A$ is an $m$-extension, then $A \cap m\N = \emptyset$. For our approach, we deal with more general sets and we define them as follows: let $m > 2$ be an integer. We say that a finite set $M \subset \N$ is an $m$-set if $[1, m-1] \subset M$ and $M \cap m\N = \emptyset$. In particular, an $m$-extension is an $m$-set. If $c$ is the largest element of an $m$-set $M$, then its depth is defined as $\left \lceil \frac{c}{m} \right \rceil$. We denote by $\mathcal{M}_\k(g)$ the set of all $m$-sets with $g$ elements, such that the maximum distance between two consecutive elements (with respect to the natural order) is $\k$ and that lies in $[1,2g - 1]$. In some cases, we also deal with subsets of $[1, 2g - 1]$ with no specific property. We denote by $\mathcal{C}_{\k}(g)$ the set of those sets that have $g$ elements and such that the maximum distance between two consecutive elements (with respect to the natural order) is $\k$. Notice that $\mathcal{G}_\k(g) \subset \mathcal{M}_\k(g) \subset \mathcal{C}_\k(g)$.

Let $g$ and $\k$ be non-negative integers and define $\phi: \mathcal{G}_{\k}(g) \to \mathcal{C}_{k+1}(g+1)$, with 

$$
G = \{1 = \l_1 < \l_2 < \cdots < \l_\a < \l_{\a+1} = \l_\a + \k < \cdots < \l_g\} \mapsto
$$
\begin{equation}
\label{mapaphi}
\mapsto \phi(G) = \{1 = \l_0 < 2 = \l_1 + 1 < \l_2 + 1 < \cdots < \l_\a + 1 < \l_{\a+1} + 2 < \cdots < \l_g + 2\},
\end{equation}
where $\a := \max\{i: \l_{i+1} - \l_i = \k\}$.

As a convention, $\phi(\emptyset) = \{1\}$ and $\phi(\{1\}) = \{1,3\}$.

\begin{theorem}
The map $\phi$ is well defined and it is injective. Moreover, if $G \in \mathcal{G}_\k(g)$ has multiplicity $m$, then 

\begin{enumerate}
  \item if $q(G) = 1$, then $\phi(G)$ is a gapset of depth $2$;
  \item if $q(G) = 2$, then $\phi(G)$ is an $(m+1)$-set of depth $2$;
  \item if $q(G) = 3$, then $\phi(G)$ is an $(m+1)$-set of depth $3$, except for the case $2m + 1 \in G$ and $\l_\a \geq 2m+1$.
\end{enumerate}
\label{phi}
\end{theorem}

\begin{proof}
First we prove that $\phi(G) \in \mathcal{C}_{\k+1}(g+1)$. The maximum distance between two consecutive elements in $\phi(G)$ is $(\l_{\a+1} + 2) - (\l_a + 1) = \k + 1$ and there are $g+1$ elements in $\phi(G)$. Since $\l_g \leq 2g -1$, then $\l_g + 2 \leq 2g+1$ and $\phi(G) \subset [1, 2(g+1) - 1]$. 

Now we prove that $\phi$ is injective.  Let $G_1 = \{a_1 < a_2 < \ldots < a_{\a_1} < a_{\a_1+1} = a_{\a_1} + \k < \ldots <a_g\}$ and $G_2 =  \{b_1 < b_2 < \ldots < b_{\a_2} < b_{\a_2+1} = b_{\a_2} + \k < \ldots <b_g\} \in \mathcal{G}_{k}(g)$ such that $\phi(G_1) = \phi(G_2)$, where $\a_1 = \max\{i: a_{i+1} - a_{i} = \k\}$ and $\a_2 = \max\{i: b_{i+1} - b_{i} = \k\}$. First, we prove that $\a_1 = \a_2$. Otherwise, without loss of generality, suppose that $\a_1 < \a_2$. Then $a_n = b_n$, for all $n \in [1,\a_1] \cup [\a_2+1, g]$ and $a_{\a_2} + 2 = b_{\a_2} + 1$. Hence, $\k = b_{\a_2+1} - b_{\a_2} = a_{\a_2+1} - (a_{\a_2}+1)$, i.e., $\k+1 = a_{\a_2+1} - a_{\a_2}$ which is a contradiction. Hence, $\a_1 = \a_2$ and $G_1 = G_2$. 

If $q = 1$, then $G$ is the ordinary gapset with genus $g$ and $\phi(G) = [1,g] \cup \{g+2\}$ is a gapset with depth 2.

It remains to solve the cases $q = 2$ and $q = 3$. Let $G = \{1 = \l_1 < \l_2 < \cdots < \l_\a < \l_{\a+1} = \l_\a + \k < \cdots < \l_g\}$ with $m(G) = m$; in particular, $[1,m-1] \subset G$ and $m \notin G$ implies that $[1,m] \subset \phi(G)$ and $m+1 \notin \phi(G)$.

If $q = 2$, then $m+1 \leq \l_g \leq 2m-1$ and the maximum element of $\phi(G)$, $\l_g + 2$, is such that $(m+1) + 1 < m+3 \leq \l_g + 2 \leq 2m+1 = 2(m+1) - 1$. Thus, $2(m+1) \notin \phi(G)$, which guarantees that $\phi(G)$ is an $(m+1)$-set with depth 2.

If $q = 3$, then $2m+1 \leq \l_g \leq 3m-1$ and the maximum element of $\phi(G)$, $\l_g + 2$, is such that $2(m+1) + 1 = 2m+3 \leq \l_g + 2 \leq 3m+1 < 3(m+1) - 1$. Thus, $3(m+1) \notin \phi(G)$. If $2m+1 \notin G$, then $2m+2 \notin \phi(G)$, since $2m$ and $2m+1 \notin G$. If $2m+1 \in G$ and $\l_\a < 2m+1$, then $2m+2 \notin G$, since the corresponding element of $2m+1$ in $\phi(G)$ is $2m + 3$. Hence, $2m+2$ and $3m+3 \notin G$ and $\phi(G)$ is an $(m+1)$-set with depth 3.
\end{proof}

\begin{remark}
In general, the image of a gapset under $\phi$ is not an $m$-set. The gapset \linebreak $A =\{1,2,3,4,6,7,9,11,14\}$ has depth 3, multiplicity $5$ and its image \linebreak $\phi(A) = \{1,2,3,4,5,7,8,10,12,16\}$ is not a $6$-set. The gapset $B = \{1, 2, 4, 5, 7, 10\}$ has depth 4, multiplicity 3 and its image $\phi(B) = \{1, 2, 3, 5, 6, 8, 12\}$ is not a $4$-set.
\end{remark}

\begin{corollary}
If $G \in \mathcal{G}_\k(g)$ has multiplicity $m$ and $2g \leq 3\k$ , then $\phi(G)$ is an $(m+1)$-set.
\label{mset}
\end{corollary}

\begin{proof}
Proposition \ref{2g3kq3} ensures that the depth of $G$ is at most 3. The cases of depth 1 or 2 follow from (1) and (2) of Theorem \ref{phi}. If $q = 3$, then Proposition \ref{abaixo} ensures that $\l_a \leq 2m-1$ and then (3) of Theorem \ref{phi} completes the proof.
\end{proof}

Eliahou and Fromentin \cite{EF} worked with the map $\sigma$ and for a gapset $G$ with canonical partition $G_0 \cup G_1 \cup G_2$ and multiplicity $m$. They define $\sigma(G) = (G_0 \cup \{m\}) \cup (G_1 + 1) \cup (G_2 + 2)$. In particular, $\sigma(G)$ is always a gapset, if its depth is at most 3. Notice that $g(\sigma(G)) = g(G) + 1$. The map $\phi$ considered in this paper is different from the map $\sigma$, besides the fact that both maps increase the genus by one. For instance, if $G = [1,9] \cup \{\textbf{11}, \textbf{19}\} \cup \{21\}$, then $\phi(G) = [1,10] \cup \{\textbf{12},\textbf{21}\} \cup \{23\}$ and $\sigma(G) = [1,10] \cup \{12,20\} \cup \{23\}$. More generally, let $g \geq 12$. If $G = [1,g-3] \cup \{g-1, 2g-5\} \cup \{2g-3\}$, then $\phi(G) = [1,g-2] \cup \{g,2g-3\} \cup \{2g-1\}$ and $\sigma(G) = [1,g-2] \cup \{g,2g-4\} \cup \{2g-1\}$.

Now we deal with pure $\k$-sparse gapsets with genus $g$ and depth 2 and 3, namely $\mathcal{G}_\k(g,2):= \{G \in\mathcal{G}_\k(g): q(G) = 2\}$ and $\mathcal{G}_\k(g,3):= \{G \in\mathcal{G}_\k(g): q(G) = 3\}$, respectively. Theorem \ref{phi} guarantees that gapsets with multiplicity $m$ and depth at most 3 are mapped onto $(m+1)$-sets with depth at most 3, except for the case of gapsets $G$ with depth 3, $2m + 1 \in G$ and $\l_\a \geq 2m+1$. Now we look for conditions that ensure that those $(m+1)$-sets are in fact gapsets.

The next result guarantees that every $m$-set with $q = 1$ or $2$ is a gapset.

\begin{lemma}
Let $G \subseteq [1, 2m-1]$ be an $m$-set. Then $G$ is a gapset with multiplicity $m$ and depth at most $2$.
\label{q=2}
\end{lemma}

\begin{proof}
Let $z \in G$ and write $z = x + y$, with $x \leq y$. Since $z \leq 2m-1$, then $x \leq m-1$. Hence, $x \in G$ and the proof is complete.
\end{proof}

\begin{theorem}
Let $g$ and $\k$ be non-negative integers and $\phi_2 := \phi \vert_{\mathcal{G}_\k(g,2)}$. Then \linebreak $\Imm(\phi_2) \subseteq \mathcal{G}_{\k+1}(g+1,2)$.
\label{q2}
\end{theorem}

\begin{proof}
Let $G' \in \Imm(\phi_2)$. Then there is $G \in \mathcal{G}_{\k}(g)$ such that $\phi_2(G) = G'$. If $m(G) = m$ and $q(G) = 2$, then Theorem \ref{phi} guarantees that $G'$ is an $(m+1)$-set with depth 2. By Lemma \ref{q=2}, we conclude that $G'$ is a gapset and $G' \in \mathcal{G}_{\k+1}(g+1,2)$.
\end{proof}

\begin{remark}
The map $\phi_2$ is injective but it is not surjective, since $[1,g] \cup \{g+2\} \notin \phi_2^{-1}(\mathcal{G}_{\k+1}(g+1,2))$. In particular, $\#\mathcal{G}_\k(g,2) < \#\mathcal{G}_{\k+1}(g+1,2)$.
\end{remark}

\begin{example}
Table \ref{table1} shows how $\phi_2$ acts.

\begin{table}
\centering
\caption{Some examples of Theorem \ref{q2} for genus 2, 3 and 4.}
\label{table1}
\begin{tabular}{c c}
\hline\noalign{\smallskip}
$G$ & $\phi_2(G)$ \\
\noalign{\smallskip}\hline\noalign{\smallskip}
$\{\textbf{1},\textbf{3}\}$ & $\{1,\textbf{2},\textbf{5}\}$ \\
$\{1,\textbf{2},\textbf{5}\}$ & $\{1,2,\textbf{3},\textbf{7}\}$ \\
$\{1,\textbf{2},\textbf{4}\}$ & $\{1,2,\textbf{3},\textbf{6}\}$ \\
$\{1,2,\textbf{3},\textbf{7}\}$ & $\{1,2,3,\textbf{4},\textbf{9}\}$ \\
$\{1,2,\textbf{3},\textbf{6}\}$ & $\{1,2,3,\textbf{4},\textbf{8}\}$ \\
$\{1,2,\textbf{3},\textbf{5}\}$ & $\{1,2,3,\textbf{4},\textbf{7}\}$ \\
$\{1,\textbf{2},\textbf{4},5\}$ & $\{1,2,\textbf{3},\textbf{6},7\}$ \\
\noalign{\smallskip}\hline
   \end{tabular}
\end{table}
\end{example}

Next, we present a result similar to Theorem \ref{q2} with an additional condition. In this case, we consider gapsets with depth 3.

\begin{theorem}
Let $g$ and $\k$ be non-negative integers such that $2g \leq 3\k$ and $\phi_3 := \phi \vert_{\mathcal{G}_\k(g,3)}$. Then $\Imm (\phi_3) \subseteq$ $\mathcal{G}_{\k+1}(g+1,3)$.
\label{q3}
\end{theorem}

\begin{proof}
Let $G = \{\l_1 < \l_2 < \cdots < \l_\a < \l_{\a+1} = \l_\a + \k < \cdots < \l_g\} \in \mathcal{G}_k(g,3)$,  where $\a = \max\{i: \l_{i+1} - \l_i = \k\}$ and $m(G) = m$. By Corollary \ref{mset}, $\phi(G)$ is an $m$-set that lies in $\mathcal{M}_{k+1}(g+1,3)$. It remains to prove that $\phi_3(G)$ is a gapset. Consider the canonical partition of $\phi(G)$, namely $G'_0 \cup G'_1 \cup G'_2$. Let $z \in \phi(G)$ and write $z = x + y$, with $x \leq y$. We consider three cases as follows:

\begin{enumerate}
  \item[$\bullet$] $z \in G'_0 = [1, m]$. In this case, both $x$ and $y \in G'_0$. 
  \item[$\bullet$] $z \in G'_1 \subseteq [m+2, 2m+1]$. In this case, $2x \leq x + y = z \leq 2m +1$. Thus, $x \leq m$ and $x \in G'_0$.
  \item[$\bullet$] $z \in G'_2 \subseteq [2m+3, 3m+1]$. We claim that if $x \leq m$ or $y \geq 2m+1$, then $x \in G'_0$. The first case is trivial and the second one implies that $3m + 1 \geq z \geq x + 2m + 1$ and we obtain $x \leq m$. Hence it remains to show that if $x, y \in [m+1, 2m]$, with $x \leq y$, then $x$ or $y \in G'_1$. 
  \begin{itemize}
  \item Consider $\l_{\a+1} \leq 2m - 1$. In this case, we claim that $y-1 \leq \l_\a$. In fact, if $y \geq \l_\a + 2$, then $z \geq (m + 1) + (\l_\a + 2) = \l_\a + m + 3$. Thus, $G \ni z-2 \geq \l_\a + m + 1$, which is a contradiction according to Proposition \ref{abaixo}. Thus, both $x-1$ and $y-1$ are smaller than or equal to $\l_\a$ and we can write $z - 2 = (x-1) + (y-1)$. Since $z-2 \in G$ and $G$ is a gapset, we conclude that $x-1$ or $y-1 \in G$ and thus $x \in \phi(G)$ or $y \in \phi(G)$. 
   \item Consider $\l_{\a+1} \geq 2m + 1$.  We use the hypothesis $2g \leq 3\k$ in this case. By Proposition \ref{gilberto}, we have $\l_\a \leq 2m - 1$ and thus $\l_\a + 1 \in G'_1$. Hence, $z \geq \l_{\a+1}+2$. In this case, $x \leq y \leq 2m < \l_{\a+1}$ and $G \ni z - 2 = (x-1) + (y-1)$. Since $G$ is a gapset, then $x-1$ or $y-1 \in G$. Hence, $x \in \phi(G)$ or $y \in \phi(G)$.
  \end{itemize}     
\end{enumerate}
\end{proof}

\begin{remark}
In general, $\phi(G)$ is not a gapset, even if the depth of $G$ is 3 and $\phi(G)$ is an $m$-set. The gapset $G =\{1,\ldots, 9, 12, 13, 14, 17, 18, 23, 28\}$ has depth 3 and multiplicity 10 and its image $\phi(G) = \{1,\ldots, 10, 13, 14, 15, 18, 19, 24, 30\}$ is an $11$-set, but it is not a gapset. There are some examples such as $G = \{1,2,4,5,7\}$ such that $\phi(G) = \{1,2,3,5,6,9\}$ is also a gapset, even when $2g > 3\k$.
\end{remark}

\begin{example}
Table \ref{table2} shows how $\phi_3$ acts.

\begin{table}
\centering
\caption{Some examples of Theorem \ref{q3} for genus 3, 4, 5 and 6.}
\label{table2}
  \begin{tabular}{c c}
\hline\noalign{\smallskip}
$G$ & $\phi_3(G)$ \\
\noalign{\smallskip}\hline\noalign{\smallskip}
$\{1,\textbf{3},\textbf{5}\}$ & $\{1,2,\textbf{4},\textbf{7}\}$ \\
$\{1,2,\textbf{4},\textbf{7}\}$ & $\{1,2,3,\textbf{5},\textbf{9}\}$ \\
$\{1,2,3,\textbf{5},\textbf{9}\}$ & $\{1,2,3,4,\textbf{6},\textbf{11}\}$ \\
$\{1,2,3,4,\textbf{6},\textbf{11}\}$ & $\{1,2,3,4,5,\textbf{7},\textbf{13}\}$ \\
$\{1,2,3,5,\textbf{6},\textbf{10}\}$ & $\{1,2,3,4,6,\textbf{7},\textbf{12}\}$ \\
$\{1,2,3,5,\textbf{7},\textbf{11}\}$ & $\{1,2,3,4,6,\textbf{8},\textbf{13}\}$ \\
$\{1,2,3,6,\textbf{7},\textbf{11}\}$ & $\{1,2,3,4,7,\textbf{8},\textbf{13}\}$ \\
\noalign{\smallskip}\hline
   \end{tabular}
\end{table}
\end{example}

\section{On the cardinality of $\mathcal{G}_\k(g)$, for $2g \leq 3\k$}
\label{s5}

In this section, we present the main result of this paper.

\begin{theorem}
Let $g$ and $\k$ be non-negative integers such that $2g \leq 3\k$. Then the map $\varphi: \mathcal{G}_\k(g) \to \mathcal{G}_{\k+1}(g+1), G \mapsto \varphi(G) = \phi(G)$ is bijective.
\label{main}
\end{theorem}

\begin{proof}
First of all, notice that the hypothesis $2g \leq 3\k$ implies that $\mathcal{G}_\k(g)$ has only the gapsets $\emptyset, \{1\}$ and gapsets with depth 2 or 3. Theorem \ref{phi} guarantees that $\varphi$ is injective; Theorems \ref{q2} and \ref{q3} guarantee that $\varphi$ is well defined. Hence, we only have to prove that $\varphi$ is surjective. 

Let $G \in \mathcal{G}_{k+1}(g+1)$ such that $2g \leq 3\k$. If $G = \{1\}$, then $G' = \emptyset$ is such that $\phi(G') = G$ and if $G = \{1,3\}$, then $G' = \{1\}$  is such that $\phi(G') = G$ . Hence, we can assume that $G$ has genus $g + 1 \geq 3$.

Write $G = \{\l_1 < \l_2 < \l_3 < \cdots < \l_\a < \l_{\a+1} = \l_\a + \k + 1 < \ldots < \l_{g+1}\}$,  where $\a = \max\{i: \l_{i+1} - \l_i = \k\}$ and assume that its multiplicity is $m$. Since $2g \leq 3\k$ then $2(g+1) \leq 3\k + 2 < 3(\k+1)$. In particular $G$ has depth at most 3 and thus $\l_{g+1} \leq 3m - 1$. Proposition \ref{unique} ensures that $\a$ is the unique index such that $\l_{\a+1} - \l_{\a} = \k + 1$ (recall $\{1,3,5\} \notin \mathcal{G}_{\k+1}(g+1)$). We show that $G' = \{\l_2 - 1 < \l_3 - 1 < \cdots < \l_\a - 1 < \l_{\a+1} - 2 < \ldots < \l_{g+1} - 2\} \in \mathcal{G}_{\k}(g)$ and satisfies $\varphi(G') = G$. 

Naturally, $\#G' = g$, the maximum distance between two consecutive elements of $G'$ is $(\l_{\a+1} - 2) - (\l_\a - 1) = \k$ and the first positive integer that does not belong to $G'$ is $m - 1$.  Now we prove that $G'$ is an $(m-1)$-set:

\begin{itemize}
  \item[$\bullet$] $m-1 \notin G'$ \\
Since $\l_j = j, \forall j \in [1,m-1]$, then $G' \ni \l_j - 1 = j - 1, \forall j \in [2,m-1]$; we also have $\l_m \geq m + 1$. We study two cases, as follows: (1) if $\l_m \geq m+2$, then $\l_m - 1 > \l_m - 2 \geq m$. In this case, we obtain that $m-1 \notin G'$; (2) suppose that $\l_m = m + 1$. If there is $j \in [m, g]$ such that $\l_{j+1} - \l_j \geq 2$, then $G' \ni \l_m - 1 = m$ and thus $m-1 \notin G'$. Otherwise, we can write $\l_j = j + 1, \forall j \in [m, g]$ and the maximum distance between two consecutive elements of $G$ is $\k + 1= 2$, which cannot occur ($2g > 3$). Hence, $m-1 \notin G'$.
  \item[$\bullet$] $2(m-1) \notin G'$  \\
Suppose that $2m - 2 \in G'$. In this case, we could have $2m \in G$ or $2m - 1 \in G$. The first case does not occur, because $m$ is the multiplicity of $G$. In the second case, we have  $2m - 1 \leq \l_\a$ (otherwise, the corresponding element would be $(2m - 1) - 2 = 2m - 3)$. By Proposition \ref{gilberto}, we have that $\l_\a = 2m - 1$. Thus, $\l_{\a+1} = \l_\a + (\k + 1) = 2m - 1 + \k + 1 \geq 2(\k+1) + \k = 3\k + 2 \geq 2(g+1)$, which is a contradiction. Hence, $2m - 2 \notin G'$.
  \item[$\bullet$] $3(m-1) \notin G'$ \\
We already know that $\l_{g+1} \leq 3m - 1$ and thus the maximum element of $G'$ is such that $\l_{g+1} - 2 \leq 3m - 3$. Hence, we only have to show that $\l_{g+1} \neq 3m - 1$ which implies that $\l_{g+1} - 2 \leq 3m - 4$. Suppose that $\l_{g+1} = 3m - 1$. In this case, we have $2m - 1 \in G$ and we could have $\l_\a < 2m - 1$ or $\l_\a = 2m - 1$. The first case implies that $\l_{\a+1} \leq 2m - 1$. By Proposition \ref{gilberto_partition}, we can denote by $\l_r$ and $\l_{r+1}$ the biggest element of $G$ that is smaller than or equal to $\l_{\a} + m$ and the smallest element of $G$ that is bigger than or equal to $\l_{\a+1} + m$, respectively. In this case, $\l_{r+1} - \l_r \geq \k$, which is a contradiction, according to Proposition \ref{unique} (notice that $G \neq \{1,3,5\}$). In the second case, we obtain $\l_{\a+1} = \l_\a + (\k + 1) = 2m - 1 + \k + 1 \geq 2(\k+1) + \k = 3\k + 2 \geq 2(g+1)$, which is a contradiction. Hence, $3m - 3 \notin G'$.
\end{itemize}

Hence, we can write $G' = G'_0 \cup G'_1 \cup G'_2$, with $G'_0 = [1,m-2]$, $G'_1 \subseteq [m, 2m-3]$ and $G'_2 \subseteq [2m-1, 3m-4]$. Now we prove that $G'$ is a gapset. Let $z \in G'$ and write $z = x + y $, with $x, y \in \N$ and $x \leq y$.

\begin{itemize}
  \item If $z \in G'_0$, then $x$ and $y \in G'_0$. 
  \item If $z \in G'_1$, then $z \leq 2m - 3$ and $2x \leq 2m - 3$. Thus $x \in G'_0$.
  \item Let $z \in G'_2$. If $x \leq m - 2$, then $x \in G'_0$. If $y \geq 2m-1$, then $x \leq m - 3$ and $x \in G'_0$. Hence, we can consider $x, y \in [m, 2m - 3]$. Notice that $\l_\a \leq 2m - 2$ (if $\l_\a = 2m - 1$, then $2m - 2 \in G'$, which does not occur). In particular, $\l_\a - 1 \leq 2m - 3$ and $z \geq \l_{\a+1} - 2$. Thus $z = \l_t - 2$, for some $t \in [\a+1, g+1]$. We claim that $y < \l_\a$. Otherwise, $z = x + y \geq m + \l_\a$ and it implies that $\l_t > \l_\a + m$, which is a contradiction according to Proposition \ref{abaixo}. Hence, $x \leq y < \l_\a$. By writing $\l_t - 2 = z = x + y$, i.e., $\l_t = (x+1) + (y+1)$ and using that $G$ is a gapset, we conclude that $x+1 \in G$ or $y+1 \in G$. Hence, $x \in G'$ or $y \in G'$.
\end{itemize}

\end{proof}

\begin{corollary}
Let $g$ and $\k$ be non-negative integers such that $2g = 3\k$. Then $\#\mathcal{G}_{\k}(g) = \#\mathcal{G}_{k+n}(g+n)$, for all $n \in \N$.
\label{mainc}
\end{corollary}

\begin{proof}
We proceed by induction on $n$. Theorem \ref{main} ensures that $\#\mathcal{G}_{\k}(g) = \#\mathcal{G}_{k+1}(g+1)$, since $2g = 3\k$. Notice that $2(g+j) \leq 3(\k + j)$, for all $j \in \N_0$ and we can apply Theorem \ref{main} again: the map $\mathcal{G}_{\k + n - 1}(g + n - 1) \to \mathcal{G}_{\k+n}(g+n), G \mapsto \phi(G)$ is bijective, which completes the proof.
\end{proof}

Next, we present instances of pure $\k$-sparse gapsets, with large $\k$.

\begin{example}
Let $g \geq 0$. The unique pure $g$-sparse gapset with genus $g$ is $[1,g-1] \cup \{2g-1\}$.
\end{example}


\begin{example}
Let $g \geq 3$. The two pure $(g-1)$-sparse gapsets with genus $g$ are $[1,g-2] \cup \{g, 2g-1\}$ and $[1,g-1] \cup \{2g-2\}$.
\end{example}

\begin{example}
Let $g \geq 6$. The five pure $(g-2)$-sparse gapsets with genus $g$ are:
\begin{itemize}
  \item[$\bullet$] $[1,g-1] \cup \{2g-3\}$
  \item[$\bullet$] $[1,g-2] \cup \{2g-4,2g-3\}$
  \item[$\bullet$] $[1,g-3] \cup \{g-1,g, 2g-2\}$
  \item[$\bullet$] $[1,g-3]\cup \{g,g+1, 2g-1\}$ 
  \item[$\bullet$] $[1,g-3] \cup \{g - 1, g+1, 2g-1\}$
\end{itemize}
\end{example}

We end up this section presenting Table \ref{g,k}, which illustrates our results. The values ware obtained using the package \texttt{numericalsgps} \cite{GAP} in GAP.

\begin{landscape}
\begin{table}[h!]
\caption{A few values for $\#\mathcal{G}_{\k}(g)$. The entries in bold correspond to the case $2g = 3\kappa$.}
\label{g,k}
\begin{tabular}{|c|c c c c c c c c c c c c c c c c c c c c| c|}
 \hline
\diagbox[height=0.6cm]{$g$}{$\k$} & $0$ & $1$ & $2$ & $3$ &$4$ & $5$ & $6$ & $7$ & 8 & 9 & 10 & 11 & 12 & 13 & 14 & 15 & 16 & 17 & 18 & 19 & $n_g$ \\
\hline
0 & \textbf{1} &  &  &  &  &  &  &  &  &  &  &  &  &  &  &  &  &  &  &  & 1 \\ 
\hline
1 &  & 1 &  &  &  &  &  &  &  &  &  &  &  &  &  &  &  &  &  &  & 1 \\
\hline
2 &  & 1 & 1 &  &  &  &  &  &  &  &  &  &  &  &  &  &  &  &  &  & 2 \\
\hline
3 &  & 1 & \textbf{2} & 1  &  &  &  &  &  &  &  &  &  &  &  &  &  &  &  &  & 4 \\
\hline
4 &  & 1 & 3 & 2 & 1 &  &  &  &  &  &  &  &  &  &  &  &  &  &  &  & 7 \\
\hline
5 &  & 1 & 5 & 3 & 2 & 1 &  &  &  &  &  &  &  &  &  &  &  &  &  &  & 12 \\
\hline
6 &  & 1 & 7 & 7 & \textbf{5} & 2 & 1 &  &  &  &  &  &  &  &  &  &  &  &  &  & 23 \\
\hline
7 &  & 1 & 10 & 12 & 8 & 5 & 2 & 1 &  &  &  &  &  &  &  &  &  &  &  &  & 39 \\
\hline
8 &  & 1 & 15 & 18 & 17 & 8 & 5 & 2 & 1 &  &  &  &  &  &  &  &  &  &  &  & 67 \\
\hline
9 &  & 1 & 20 & 31 & 28 & 18 & \textbf{12} & 5 & 2 & 1 &  &  &  &  &  &  &  &  &  &  & 118 \\
\hline
10 &  & 1 & 27 & 51 & 49 & 34 & 22 & 12 & 5 & 2 & 1 &  &  &  &  &  &  &  &  &  & 204 \\
\hline
11 &  & 1 & 38 & 78 & 87 & 57 & 40 & 22 & 12 & 5 & 2 & 1 &  &  &  &  &  &  &  &  & 343 \\
\hline
12 &  & 1 & 51 & 125 & 147 & 100 & 76 & 42 & \textbf{30} & 12 & 5 &  2 & 1 &  &  &  &  &  &  &  & 592 \\
\hline
13 &  & 1 & 70 & 195 & 237 & 177 & 134 & 83 & 54 & 30 & 12 & 5 & 2 & 1 &  &  &  &   &  &   & 1001 \\
\hline
14 &  & 1 & 95 & 297 & 399 & 309 & 239 & 150 & 99 & 54 & 30 & 12 & 5 & 2 & 1 &  &  &  &  &  & 1693 \\
\hline
15 &  & 1 & 128 & 457 & 654 & 530 & 422 & 259 & 183 & 103 & \textbf{70} & 30 & 12 & 5 & 2 & 1 &  &  &  &  & 2857 \\
\hline
16 &  & 1 & 172 & 705 & 1061 & 902 & 723 & 452 & 336 & 199 & 135 & 70 & 30 & 12 & 5 & 2 & 1 &  &  &  & 4806 \\
\hline
17 &  & 1 & 230 & 1074 & 1717 & 1513 & 1248 & 811 & 590 & 363 & 243 & 135 & 70 & 30 & 12 & 5 & 2 & 1 &  &  & 8045 \\
\hline
18 &  & 1 & 309 & 1621 & 2777 & 2535 & 2148 & 1411 & 1037 & 646 & 444 & 251 & \textbf{167} & 70 & 30 & 12 & 5 & 2 & 1 &  & 13467 \\
\hline
19	  & & 1 & 413 & 2448 & 4464 & 4232 & 3636 & 2434 & 1810 & 1124 & 804 & 480 & 331 & 167 & 70 & 30 & 12 & 5 & 2 & 1 & 22464 \\
\hline
\end{tabular}
\end{table}
\end{landscape}

\section{Further questions}

A natural question that arises from Theorem \ref{main} is about the existence of an injective map from $\mathcal{G}_\k(g)$ to $\mathcal{G}_{\k+1}(g+1)$ where $\k$ and $g \in \N_0$ and with no other hypothesis. The values obtained in Table \ref{g,k} point out that it is possible that it occurs. Besides, those values indicate that it is possible to have an injective map from $\mathcal{G}_\k(g)$ to $\mathcal{G}_{\k}(g+1)$. We can write those questions as follows:

\begin{question}
Let $g$ and $\k$ be non-negative integers. Is there an injective map \linebreak $\mathcal{G}_\k(g) \to \mathcal{G}_{\k+1}(g+1)$?
\end{question}

\begin{question}
Let $g$ and $\k$ be non-negative integers. Is there an injective map \linebreak $\mathcal{G}_\k(g) \to \mathcal{G}_{\k}(g+1)$?
\end{question}

Notice that a positive answer to any of them implies that the sequence $(n_g)$ is increasing for $g \in \N$.

Another question that arises is about the behaviour of the sequence $(g_w)$, where $g_w := \#\mathcal{G}_{2w}(3w)$. In fact, Corollary \ref{mainc} ensures that $\#\mathcal{G}_{\k}(g) = \#\mathcal{G}_{\k+n}(g+n)$, for all $n \in \N$, as can be seen in Table \ref{g,k} where columns progressively stabilize below the diagonal. Here are the first few values of the sequence $(g_w)$:
$$(g_0, g_1, g_2, g_3, g_4, g_5, g_6) = (1, 2, 5, 12, 30, 70, 167).$$ 
Hence, it could be of interest to know information about the sequence, which was unknown at OEIS before this work, but which is now listed there as sequence A348619.

\begin{table}
\centering
\caption{Some sequences related to $(g_w)$}
\begin{tabular}{c c c c}
\hline
$w$ & $g_w$ & $\frac{g_{w}}{g_{w-1}}$ & $\frac{\sum_{i=0}^{w} g_{i}}{g_{w}}$ 
\\
\noalign{\smallskip}\hline\noalign{\smallskip}
0 & 1 & $-$ & 1  \\
\hline
1 & 2 & 2.000 & 1.5  \\
\hline
2 & 5 & 2.500 & 1.6 \\
\hline
3 & 12 & 2.400 & 1.66 \\
\hline
4 & 30 & 2.500 & 1.66 \\
\hline
5 & 70 & 2.333 & 1.714 \\
\hline
6 & 167 &  2.386 & 1.719 \\
\hline
7 & 395 & 2.365 & 1.727 \\
\hline
8 & 936 & 2.370 & 1.729 \\
\hline
9 & 2212 & 2.363 & 1.731 \\
\noalign{\smallskip}\hline
   \end{tabular}
\end{table}

\begin{question}
Are there remarkable properties about the sequence $(g_w)$?
\end{question}

\textbf{Acknowledgements.} The authors thank to the anonymous referee and to Nathan Kaplan for their careful corrections, suggestions and comments that allowed to improve the last version of the paper.

%
%



\end{document}